\documentclass[12pt]{amsart}

\usepackage[english]{babel} 
\usepackage[utf8]{inputenc}

\usepackage{hyperref}
\usepackage[capitalize]{cleveref}
\usepackage{tikz-cd}
\usepackage{enumitem}
\usepackage{amsfonts}
\usepackage{amsmath}
\usepackage{amssymb}
\usepackage{amsthm}
\usepackage{mathrsfs}
\usepackage{mathtools}
\usepackage{stmaryrd}
\hypersetup{
    colorlinks=false,
    linkcolor=blue
}

\usepackage{tikz-cd}

\usepackage{graphicx}
\usepackage[letterpaper, margin=1in, total={6in, 8in}]{geometry}

\newtheorem{thm}{Theorem}[section]

\newtheorem{cor}[thm]{Corollary}
\newtheorem{prop}[thm]{Proposition}

\theoremstyle{definition}

\newtheorem{exmp}[thm]{Example}

\theoremstyle{remark}
\newtheorem{rem}[thm]{Remark}

\newcommand{\N}{\mathbf N}
\newcommand{\Z}{\mathbf Z}
\newcommand{\Q}{\mathbf Q}

\newcommand{\C}{\mathbf C}
\newcommand{\A}{\mathbf A}

\newcommand{\Gal}{\operatorname{Gal}}

\newcommand{\FF}{\mathbf{F}}
\newcommand{\PP}{\mathbf{P}}

\newcommand{\SL}{\operatorname{SL}}

\newcommand{\GL}{\operatorname{GL}}
\renewcommand{\H}{\mathcal{H}}
\renewcommand{\mod}{\operatorname{mod }}

\newcommand{\Hom}{\operatorname{Hom}}

\newcommand{\End}{\operatorname{End}}

\newcommand{\calF}{\mathcal{F}}

\newcommand{\calO}{\mathcal{O}}

\usepackage[OT2,T1]{fontenc}
\DeclareSymbolFont{cyrletters}{OT2}{wncyr}{m}{n}
\DeclareMathSymbol{\Sha}{\mathalpha}{cyrletters}{"58}
\nocite{*}

\title{Rigorous expansions of modular forms at CM points, I: Denominators}
\author{Chris Xu}
\date{\today}
\email{chx007@ucsd.edu}
\address{Department of Mathematics, University of California San Diego, La Jolla, California 92093, USA}

\begin{document}

\maketitle
\begin{abstract}
    We describe an algorithm to rigorously compute the power series expansion at a CM point of a weight $2$ cusp form of level coprime to $6$. Our algorithm works by bounding the denominators that appear due to ramification, and without recourse to computing an explicit model of the corresponding modular curve. Our result is the first in a series of papers toward an eventual implementation of equationless Chabauty.

\end{abstract}

\section{Introduction}
Let $X_H$ be a modular curve of some level $H \leq \GL_2(\Z/N\Z)$, let $\tau \in \H$ be a CM point in the upper-half plane, and let $\omega := f(q)\,dq$ be a $1$-form that corresponds to the weight $2$ cusp form $qf(q)$. We will describe how to compute a power series expansion for $f(q)\,dq$ at the point $\tau$. 

Such a computation can be done analytically without too much difficulty: see \cite{voightpowerseries} and \cite{CKL}. But if $\omega$ is defined over $\Z$, and we choose an algebraic uniformizer $t$, then the coefficients of the resulting power series $g(t)\,dt$ will actually lie in a number field $F$. This article is the first part of a two-paper series on how to \emph{rigorously} pin down the coefficients as elements of $F$. As far as we know, none of the literature has ever detailed such an algorithm.

We will implement the algorithm outlined in this paper when the second part, \cite{huangrendellxu}, is released. Eventually we hope to use it for Chabauty computations on modular curves.

\subsection{Future work}
This paper started off as a collaboration with Yongyuan Huang and Isabel Rendell for a larger project. Namely, we wanted to develop a general-purpose Chabauty algorithm for modular curves. During the collaboration, it became clear that in order to make further headway, we needed to find a way to systematically compute power series of cusp forms at possibly non-cuspidal points. The paper here provides my ideas on how to approach some (but not all) aspects of the computation. In a joint work in progress (cf. \cite{huangrendellxu}), we will detail the rest of the algorithm; more precisely, we will explain how to perform a rigorous precision analysis and how to compute cusp forms that correspond to differentials defined over $\Z$.

\subsection{Leitfaden}
This article is split up into the following parts. \cref{section:notationsetup} sets up the notation surrounding the basepoint $\tau$. \cref{section:analyticcomputation} recalls how to analytically compute the power series coefficients. \cref{section:ramificationoverview} describes a strategy to bound the denominators appearing in the power series coefficients in terms of the ramification experienced at $\tau$. \cref{Primes dividing $N$} and \cref{Primes dividing $j_E$ or $j_E - 1728$} give algorithms for computing the necessary denominator bounds. \cref{section:0or1728} briefly describes the changes needed to the algorithm if $j(\tau) \in \{0,1728\}$. Finally, \cref{alg:algintegerrecovery} details how to recover an algebraic integer in a prescribed number field, given approximations to its conjugates.

\subsection{Acknowledgements}
I thank Yongyuan Huang and Isabel Rendell for their fruitful collaboration. I thank Mingjie Chen and Jun Lau for sharing their code for the computations done in \cite{CKL}. I thank Jennifer Balakrishnan, K{\c e}stutis {\v C}esnavi{\v c}ius, Brian Conrad, Maarten Derickx, Sachi Hashimoto, Kiran Kedlaya, Travis Morrison, Bjorn Poonen, Andrew Sutherland and John Voight for helpful conversations.
\section{Notation and setup} \label{section:notationsetup}
\subsection{Modular curves}
Let $H \leq \GL_2(\Z/N\Z)$. For a scheme $S$, let $Y_{H,S}/S$ denote the coarse moduli scheme of elliptic curves $E/S$ that are endowed with an $H$-level structure $[\iota]_H$. Denote $X_{H,S}$ as $Y_{H,S}$ plus cusps. Let $\pi\colon X_{H,S} \to \PP^1_S$ denote the map to $X(1)_S \cong \PP^1_S$ that forgets the level structure.
Fix a Drinfeld basis $(\Z/N\Z)^2 \to E(S)$ of $E$, possibly after passing to an fppf cover. Under this basis we may identify $[\iota]_H$ with a double coset $HgA_E$ of $\GL_2(\Z/N\Z)$; here $A_E$ denotes the group $\underline{\text{Aut}}_E(S)$ of fppf local automorphisms. For every elliptic curve we encounter with an $H$-level structure, assume we have chosen such a basis so that we are free to use the double coset characterization as we please. If we specify a particular basis, we will always use that basis when talking about double cosets. (See e.g. \cite[2.3]{RSZB} for a precise definition of $H$-level structures.) 

Assume always that $-I \in H$ and that $\det\colon H \to (\Z/N\Z)^\times$ is surjective. Write $Y_{H,\Z}$ as just $Y_H$. Finally, assume that $\gcd(6,N) = 1$. This last condition is justified by the following statement which we will use implicitly in our analysis of denominators:
\begin{prop}[{\cite[Prop. 6.4]{Ces16}}]
    For any $\Z[1/\gcd(6,N)]$-scheme $S$, the canonical ``coarse base change'' map $Y_{H,S} \to Y_H \times S$ is an isomorphism.
\end{prop}

\begin{rem}
    By \cite[Thm 6.12]{rydhalgstack}, $Y_{H,S} \to Y_H \times S$ is a universal homeomorphism for any scheme $S$. Hence the failure of coarse base change is rather mild at worst.
\end{rem}

\subsection{Basepoint}
Choose a CM point $b_\infty := (E,[\iota]_H) \in Y_H(\C)$. The corresponding elliptic curve $E$ has complex multiplication by an order $R := \Z[\tau]$, where $\tau$ is an imaginary quadratic integer. Let $P(T) := T^2 + rT + s$ denote the characteristic polynomial of $\tau$. Without loss of generality, we may assume that $E_\C \cong \C/R$, as all other choices of $E$ will be Galois conjugates. Let $j_E\in \C$ denote the $j$-invariant of $E$, and fix the basis $[\tau/N,1/N]$ for $E[N](\C)$.

By CM theory, $E$ may be defined over the ring class field $F_R$ of $R$. Denote $F_{R, N}\coloneq F_{R}(E[N])$ the field obtained by adjoining all $N$-torsion points of $E$ to $F_R$. Choose an embedding $F_{R,N} \hookrightarrow \C$ and regard $F_{R,N}$ as a subfield of $\C$ this way. Assume that $\tau$ has positive imaginary part under this embedding.  For each prime $p$, fix an embedding $F_{R,N} \hookrightarrow \overline{\Q}_p$, denote $F_{R,N,p}$ the completion of its image, and denote $\mathfrak{p}_N$ the prime of $\calO_{F_{R, N}}$ above $p$ that corresponds to this embedding. Define $\mathfrak{p} := \mathfrak{p}_1$, define $F_{R,p} := F_{R,1,p}$, and let $\pi_{R,p}$ be a uniformizer of $\calO_{F_{R,p}}$. Refer to the Zariski closure of $b_\infty$ in $X_H$ as just $b$. Refer to the mod $\mathfrak{p}_N$ reduction of $b$ as $b_p$.

\subsection{Good reduction models of $E$ at each prime}
Form the model of $E$ given by $$E_0 \colon y^2+xy = x^3 - \frac{36}{j_E-1728}x - \frac{1}{j_E-1728}.$$ Since $\Delta(E_0) = j_E^3/(j_E-1728)^3$, $E_0$ has bad reduction at precisely the primes dividing $j_E(j_E-1728)$. But $E_0$ still has everywhere potentially good reduction because it is CM.

Identify the periods of $E_0(\C)$ with the lattice $\Lambda \subseteq \C$ with respect to the N\'eron differential $dx/(2y+x)$; do this while looping over possible embeddings $F_R \hookrightarrow \C$, until $\Lambda$ can be identified with $\Z\tau+\Z$ after some scaling. The end result is a complex analytic isomorphism $\C/(\Z\tau+\Z) \xrightarrow{\sim} E_0(\C)$. Keep the embedding $F_R \hookrightarrow \C$ we just found.
 
For each rational prime $p$ such that $\mathfrak{p}|j_E(j_E-1728)$, the elliptic curve $E_0$ attains good reduction after passing to the degree $e_p$ tamely ramified extension $F_{R,p,(e_p)} := F_{R,p}\left(\pi_{R,p}^{1/e_p}\right)$ for some $e_p$. By \cite{timandvlad}, we may determine $e_p$ by the Kodaira symbol of $E_0$ at $\mathfrak{p}$ as follows:
\begin{center}
\begin{tabular}{ c|c|c|c|c}
    Kodaira symbol & $II^{(*)}$ & $III^{(*)}$ & $IV^{(*)}$ & $I_0^*$\\
    \hline 
    $e_p$ & 6 & 4 & 3 & 2
\end{tabular}
\end{center}
Denote $E_{\calO,p}$ a good reduction model of $E_0$ defined over $\calO_{F_{R,p,(e_p)}}$, and choose an isomorphism $\iota_p\colon E_0 \cong E_{\calO,p}$ over $F_{R,p,(e_p)}$. Let $E_{\FF_{\mathfrak{p}}}$ denote the special fiber of $E_{\calO,p}$. Let $F_{R,N,p,(e_p)}$ denote a compositum $F_{R,N,p}F_{R,p,(e_p)}$.

If rational prime $p$ is such that $\mathfrak{p}$ does not divide $j_E(j_E-1728)$, then denote $E_{\calO,p}$ simply by the base change $(E_0)_{\calO_{F_{R,p}}}$, and denote $\iota_p$ accordingly.  

\subsection{Level structure}
Let $\iota\colon (\Z/N\Z)^2 \to E[N](\calO_{F_{R, N}})$ represent the $H$-level structure on $E$ corresponding to $b$. If $N = \prod_i p_i^{m_i}$ denotes the prime factorization of $N$, let us make the identification 
$Y(N)(\C)^{an} \cong \GL_2(N) \times_{\SL_2(\Z)} \H$ 
so that the tuple $(g,\tau) \in X(N)(\C)^{an}$ will correspond to the level structure $(E_\tau,\iota)$, where $E_\tau = \C/(\Z+\Z\tau)$ and $\iota\colon (\Z/N\Z)^2 \to E[N](\C)$ is given by $\iota(a,b) = g \cdot \left((a\tau+b)/N\right)$. Denote this $\iota$ as $g \cdot [\tau,1]$ for short. Denote $q_b :=q(b) := \exp(2\pi i \tau) \in \C$.


\section{Computation of power series coefficients} \label{section:analyticcomputation}
Let us first recall our setting (c.f. \cite[Algorithm 3.5 Step 1]{CKL}), which provides a recursive method to compute analytically the coefficients in the power series expansion of a cusp form at a given point on a modular curve. Fix a differential $\omega\in H^0(X_H,\Omega^1)$ with $q$-expansion $f(q)\,dq \in \Z[\zeta_N][[q]]\,dq$. (That is, we presume $\omega$ is defined over $\Z$.) The differential $\omega$ corresponds to the weight $2$ cusp form $qf(q)$. Denoting $j\coloneq j(q)\in q^{-1}\Z[[q]]$ as the $j$-invariant, we would now like to express $f(q)\,dq$ in terms of the parameter $t \coloneq j-j_E$. That is, we must solve for the coefficients $\mathbf{c}\coloneq [c_0,\ldots,c_n]^T$ in the equation $f(q)\,dq = \sum\limits_{\ell \geq 0}^\infty c_\ell t^\ell\,dt =: g(t)\,dt$
for arbitrarily high $n\in\N$. 

First, express both $f(q)\,dq$ and $t$ as power series with respect to the parameter $q-q_b$: that is, let $a_{\ell}=\frac{j^{(\ell)}(q_b)}{\ell!}\text{ and }b_{\ell}=\frac{f^{(\ell)}(q_b)}{\ell!}$, so that we have expansions $$f(q)\,dq = \sum\limits_{\ell \geq 0}^\infty b_\ell (q-q_b)^\ell\,dq \text{ and }
    t = \sum\limits_{\ell \geq 1}^\infty a_\ell (q-q_b)^\ell.$$
We obtain the equality $\sum\limits_{\ell \geq 0}^\infty b_\ell(q-q_b)^\ell\,dq = \sum\limits_{\ell \geq 0}^\infty c_\ell t^\ell$. Substituting our expression for $t$ into the right hand side of this equality and then comparing coefficients on both sides, we arrive at the following characterization of $\mathbf{c}$:
\begin{prop}
    Let $\mathbf{b}$ denote the column vector $[b_0, \dots, b_n]^T$. We have $\mathbf{c} = M^{-1}\mathbf{b}$, where $M$ is the $0$-indexed lower triangular matrix whose entries satisfy the recurrence relation
    \begin{align*}
        M_{\ell, j} &= \begin{cases}
            (\ell+1)a_{\ell+1} & j=0 \\
            a_1M_{\ell-1,j-1} + a_2M_{\ell-2,j-1} + \cdots + a_{\ell-j+1}M_{j-1,j-1} & 1 \leq j \leq \ell
        \end{cases}
    \end{align*}
    for all $0 \leq \ell \leq n$. Moreover, the coefficients $c_\ell$ all lie in $F_{R,N}$.
\end{prop}

\section{Ramification as the source of denominators} \label{section:ramificationoverview}
Since $f(q)$ corresponds to a differential defined over $\Z$, one would a priori expect the coefficients $c_\ell$ of $g(t)$ to lie in $\calO_{F_{R,N}}$. Unfortunately, this does not take into account the fact that the map $\pi\colon Y_H\to\A^1$ may be ramified below various reductions $b_p$ of our point $b$. In the next few sections, we explain how to construct a product of rational exponent prime powers $C := \prod_i p_i^{r_i} \in \overline{\Z}$ such that the substitution $u:=t/C$ guarantees $g(t)\,dt =: \tilde{g}(u)\,du\in\overline{\Z}[[u]]\,du$. If we denote $C^{[n]}$ by $C^{[n]} := \prod_i p_i^{\lceil n\cdot r_i \rceil}$ for each $n \geq 1$, we will in fact have $c_{\ell} \in \frac{1}{C^{[\ell+1]}}\calO_{F_{R,N}}$. 

Let $\{\sigma_1, \dots, \sigma_d\}$ denote the elements of $\Gal(F_{R,N}/K)$. Repeating the computations in \cref{section:analyticcomputation} for all points $\sigma_j(\tau_b) \in \H$, we get for each $\ell$ a list $$\mathbf{c}_\ell := [c_\ell C^{[\ell+1]}, \sigma_2(c_\ell) C^{[\ell+1]}, \dots, \sigma_d(c_\ell) C^{[\ell+1]}].$$ Applying \cref{alg:algintegerrecovery} to each $\mathbf{c}_\ell$ allows us to rigorously pin down $c_\ell C^{[\ell+1]}$ as an algebraic integer in $F_{R,N}$ and ergo determine $c_\ell \in F_{R,N}$. But to do this, we will need to first compute $C$.
\begin{rem}
    Because we are working with an $H$-level structure and not full level $N$ structure, the point $b_\infty$ will actually be defined over the smaller subfield $F_{R,N}^{\Gal(F_{R,N}/K) \cap H}$. So we do not have as many points $\sigma_j(\tau_b)$ that we need to repeat the \cref{section:analyticcomputation} computations for.
\end{rem}

\subsection{Ramification scenarios}
Let us first explicate how ramification over the $j$-line occurs. Recall that the ramification points of $Y_{H,\C}$ are precisely the tuples $(E, [\iota]_H)$ such that $j(E) \in \{0,1728\}$, and such that $[\iota]_H =: HgA_E$ satisfies $A_E \nsubseteq g^{-1}Hg$. Passing to integral models, we find the following two scenarios:
\begin{itemize}
    \item Ramification on a horizontal divisor: If $b_\infty$ lies in the same mod $\mathfrak{p}_N$ residue disk as a ramification point of $X_{H,\C}$, then the reduction $b_p$ is ramified over the $j$-line. For this to occur, it is necessary, but not sufficient, to have $\mathfrak{p}|j_E(j_E-1728)$.

    \item Ramification on a vertical divisor: If $b_\infty$ shares the same mod $\mathfrak{p}_N$ residue disk as another point $b' \in X_{H,\C}$ such that $j(b_\infty) = j(b')$, then $b_{p}$ is ramified over the $j$-line. For this to occur, it is necessary, but not sufficient, to have $p|N$.
\end{itemize}
As a result of this discussion, we have the following:
\begin{prop}
    Suppose that $\pi\colon Y_H \to \A^1$ is ramified at $b_p \in Y_H(\FF_{\mathfrak{p}})$. Then $\mathfrak{p}$ divides $Nj_E(j_E-1728)$.
\end{prop}
\begin{rem}
    Note that the two scenarios above are not mutually exclusive. In particular, a CM point will tend to have many small prime divisors (cf. \cite{grosszagier}). So it will be common for $N$ to share a factor with $j_E$ or $j_E-1728$, and therefore it is probable for both horizontal and vertical ramification to occur.
\end{rem}
We will work one prime at a time to eliminate any ramification we encounter. Thus, from now on, fix a rational prime $p$ such that either $p|N$ and/or $\mathfrak{p}|j_E(j_E-1728)$.
\subsection{Eliminating ramification: the smooth case}\label{subsec:eliminatingramification}
Let us further explicate on the nature of ramification on the map $\pi\colon Y_H \to \A^1$. Denote $K := F_{R,N,p}$ for short. Let $\nu(\cdot)$ and $|\cdot|$ denote the valuation and absolute value functions, normalized so $\nu(p)=1$. Denote $\widehat{Y}_{H,b}$ (resp. $\widehat{\A}^1_t$) the formal completions of $Y_H \otimes \calO_{F_{R,N,p}}$ (resp. $\A^1 \otimes \calO_{F_{R,N,p}}$) with respect to the divisors $b$ (resp. $t = 0$). In this part let us assume that the special fiber $Y_{H,\FF_{\mathfrak{p}_N}}$ is \emph{smooth} at $b_p$; in the next subsection we will explain the modifications needed to make our arguments here work for the general case.

\begin{rem}
    The arguments here may be thought of as an effective version of the ``reduced fiber theorem'' (cf. \cite{reducedfibre}) in the simplest case.
\end{rem}

Fix a formal parameter $x$ on $\widehat{Y}_{H,b}$ such that $x(b) = 0$. By the smoothness assumption, $\pi\colon \widehat{Y}_{H,b} \to \widehat{\A}^1_t$ may be described as the vanishing locus of some $f(x) \in \calO_K[[t]][x]$ as $x$ and $t$ are allowed to vary in an open unit disk, i.e. 
$$ f(t,x) := x^N + c_{N-1}(t)x^{N-1} + \cdots + c_1(t)x + c_0(t) = 0\quad(|x|<1,~|t|<1). $$
In particular, on rigid analytic generic fibers, the map $\pi$ is just a map of open unit disks $\mathbf{D}(1,x) \to \mathbf{D}(1,t)$. 

Let $D_{vert} \subset \mathbf{D}(1,x)$ denote the (possibly empty) locus $\pi^{-1}(\pi(b))\backslash \{b\}$, and let $D_{horiz} \subset \mathbf{D}(1,t)$ denote the branch locus of $\pi$ on the target; concretely, $D_{horiz}$ is the $t'$ for which $f(t',x)$ has a multiple root. Let $e :=1+ \# D_{vert}$. Denote the quantities 
$$v_{horiz} := \max_{t' \in D_{horiz}} \nu(t'),\quad v_{vert} := \max_{x' \in D_{vert}} \nu(x').$$
By convention let $v_{horiz} = 0$ if $D_{horiz}$ is empty.

We can now immediately eliminate the horizontal ramification when we shrink the target to a small enough radius. The following is clear from the definitions:
\begin{prop}
    Above the open subdisk $\mathbf{D}(|p^{v_{horiz}}|,t) \subset \mathbf{D}(1,t)$, the map $\pi\colon \widehat{Y}_{H,b} \to \widehat{\A}^1_t$ does not observe any ramification coming from a horizontal divisor of $Y_H$.
\end{prop}
Eliminating the vertical ramification is harder. We first observe another immediate consequence of the definitions.
\begin{prop}
    The disk $\mathbf{D}(|p^{v_{vert}}|,x)$ does not intersect $D_{vert}$.
\end{prop}
Now we claim the following.
\begin{prop} \label{prop:onetoone}
    For each $t' \in \mathbf{D}(|p^{e \cdot v_{vert}}|,t)$, there is exactly one point $x' \in \mathbf{D}(|p^{v_{vert}}|,x)$ that maps to $t'$ under $\pi$.
\end{prop}
\begin{proof}
    For $t_0 \in \C_p$, let $N_{f,t_0}$ denote the Newton polygon of $f(t_0,x) \in \C_p[x]$. Because $D_{vert} \cup \{b\}$ describes precisely the points above $t=0$ and has cardinality $e$, we find that $(e,0)$ must necessarily be a vertex of $N_{f,0}$, and that all vertices before $(e,0)$ must lie above the horizontal axis. We also note that $v_{vert}$ describes the largest finite slope of $N_{f,0}$. Combining the above two observations together, we find that the first vertex of $N_{f,0}$ is $(1,v_1)$ where crucially we have $v_1 \leq (e-1)\cdot v_{vert}$.

    We have now established that all vertices of $N_{f,0}$ have $y$-coordinate at most $(e-1)\cdot v_{vert}$. Now consider $t$ such that $|t| < |p^{e\cdot v_{vert}}|$. We find two things:
    \begin{itemize}
        \item The Newton polygon $N_{f,t}$ is the exact same as $N_{f,0}$ from $(1,v_1)$ to $(e,0)$. In particular, all $y$-coordinates of $N_{f,0}$ are below $e\cdot v_{vert}$, so replacing $t=0$ with a $t$ satisfying $|t| < |p^{e\cdot v_{vert}}|$ does not change the coefficient valuations of $c_i(t)$, for the $i$ that correspond to vertices of $N_{f,0}$. 

        \item The Newton polygon $N_{f,t}$ has an extra coordinate of the form $(0,e\cdot v_{vert}+\epsilon)$ for some $\epsilon > 0$.
    \end{itemize}
    We thus find that the first edge of $N_{f,t}$, corresponding to the segment between $(0,e\cdot v_{vert}+\epsilon)$ and $(1,v_1)$, has slope strictly greater than $v$. On the other hand, all subsequent edges have slope at most $v$.

    We conclude that for $t' \in \mathbf{D}(|p^{e \cdot v_{vert}}|,t)$, the polynomial $f(t',x) = 0$ has precisely one root $x$ of slope greater than $v$. The lemma follows.
\end{proof}
\begin{cor}
    Above the open subdisk $\mathbf{D}(|p^{e \cdot v_{vert}}|,t) \subset \mathbf{D}(1,t)$, the map $\pi\colon \widehat{Y}_{H,b} \to \widehat{\A}^1_t$ does not observe any ramification coming from a vertical divisor of $Y_H$.
\end{cor}

\begin{cor}
    Let $v_{opt} := \max(v_{horiz},e \cdot v_{vert})$. Above the open subdisk $\mathbf{D}(|p^{v_{opt}}|,t) \subset \mathbf{D}(1,t)$, the map $\pi\colon \widehat{Y}_{H,b} \to \widehat{\A}^1_t$ does not observe any ramification. Namely, we have an isomorphism 
    $$\pi^{-1}\left(\mathbf{D}(|p^{v_{opt}}|,t)\right) \cap \mathbf{D}(1,x) \xrightarrow{\sim} \mathbf{D}(|p^{v_{opt}}|,t).$$
\end{cor}
\subsection{Eliminating ramification: the singular case}
Assume the special fiber $Y_{H,\FF_{\mathfrak{p}_N}}$ is not smooth at $b_p$. The generic fiber of $\widehat{Y}_{H,b}$ will no longer be a unit disk, but we can still embed it into an open unit polydisk, with formal parameters $\mathbf{x} = (x_1, \dots, x_n)$. The statements involving $D_{horiz}$, $v_{horiz}$, $D_{vert}$ and $v_{vert}$ still go through \emph{mutatis mutandis}. 

To finish off, let us use the following trick. First, note that the adapted statements give us an open polydisk $\mathbf{D}(|p^{v_{vert}}|, \mathbf{x}) := \prod\limits_{i=1}^n\mathbf{D}(|p^{v_{vert}}|,x_i)$ disjoint from $D_{vert}$. Next, we may choose a conformal isomorphism $z\colon \mathbf{D}(1,\mathbf{x}) \xrightarrow{\sim} \mathbf{D}(1,\mathbf{y})$ of $n$-dimensional open unit polydisks such that:
\begin{itemize}
    \item The isomorphism maps $\mathbf{D}(|p^{v_{vert}}|, \mathbf{x})$ onto $\mathbf{D}(|p^{v_{vert}}|, \mathbf{y})$.

    \item For each projection $\pi_j\colon \mathbf{D}(1,\mathbf{y}) \to \mathbf{D}(1,y_j)$, the images $(\pi_j \circ z)(\mathbf{D}(|p^{v_{vert}}|, \mathbf{x}))$ and $(\pi_j \circ z)(D_{vert})$ are disjoint from each other.
\end{itemize}
From the above, we obtain $n$ polynomials $f_j(t,y_j) \in \calO_{\C_p}[[t]][y_j]$ satisfying exactly the conditions as in the proof of \cref{prop:onetoone}. So we may again restrict the target of our map $\pi$ to $\mathbf{D}(|p^{e\cdot v_{vert}}|,t)$ and get the same result as before: the map $\pi$ is unramified above $\mathbf{D}(|p^{e\cdot v_{vert}}|,t)$. The ensuing corollaries follow immediately.

\subsection{Effects of eliminating ramification on our power series}
Let $v_{den}$ denote the smallest rational number such that the substitution $u:=t/p^{v_{den}}$ guarantees that $g(t)\,dt =: \tilde{g}(u)\,du$ is $p$-integral. Because we chose our differential $\omega$ to be defined over $\Z$, we have a somewhat stronger than expected condition for $p$-integrality:
\begin{prop}\label{prop:mustcheckallgaloisconjugates}
    We have $v_{den} = \max \{v_{opt}(\sigma(b))\colon \sigma \in \Gal(F_{R,N}/F_R)\}$, the maximum of $v_{opt}$ for points $\sigma(b)$, as $\sigma(b)$ ranges across Galois conjugates of $b$ over $F_R$.
\end{prop}
The key point behind \cref{prop:mustcheckallgaloisconjugates} is that the action of $\sigma \in \Gal(F_{R,N}/F_R)$ on the coefficients of $g(t)\,dt$ preserves the $p$-adic valuations of the $c_\ell$, and so ramification at any one of the conjugates means denominators for all conjugates.

In light of \cref{prop:mustcheckallgaloisconjugates}, define the quantities
\begin{align*}
    v_{d,horiz} &:= \max \{v_{horiz}(\sigma(b))\colon \sigma \in \Gal(F_{R,N}/F_R)\} \\
    v_{d,vert} &:= \max \{v_{vert}(\sigma(b))\colon \sigma \in \Gal(F_{R,N}/F_R)\} \\
    e_{den} &:= \max \{e(\sigma(b))\colon \sigma \in \Gal(F_{R,N}/F_R)\}.
\end{align*}

In the next two sections, we will analyze each ramification scenario, and for each relevant prime $\mathfrak{p}|p$ provide an algorithm to compute $v_{den}$. Taking the product of the prime powers $p^{v_{den}}$ for each prime $p$ encountered will then give us $C$.

\section{Denominators coming from vertical ramification} \label{Primes dividing $N$}
Suppose that $p|N$, but that $\mathfrak{p}$ does not divide $j_E(j_E-1728)$. Let $e$, $v_{horiz}$, $v_{vert}$, $v_{den}$, $e_{den}$, $e_{d,horiz}$ and $e_{d,vert}$ be as in \cref{section:ramificationoverview}; note that we necessarily have $v_{horiz} = 0$, so $v_{opt} = e\cdot v_{vert}$. 

\subsection{The Newton polygon attached to the formal group law}
On $E_{\calO, p}$, fix a parameter $T$ at the zero section to obtain a formal group law $\hat{E}$. For $n \in \Z$ let $[n](T) = nT + O(T^2)$ denote the power series induced by the multiplication by $n$ map. Recall that the formal group law has height $h := \text{ht}(\hat{E})$ precisely when $[p](T) \equiv T^{p^h} \,\mod \,(\pi,T^{p^h+1})$.

Write $N$ as $N_0p^m$, where $N_0$ is coprime to $p$. For an ideal $I \leq \calO_{F_{R,N,p}}$, let $E_{\calO,p,I} := E_{\calO,p} \otimes \calO_{F_{R,N,p}}/I$. The Newton polygon of $[p^m](T)$ is then an important invariant controlling the group schemes $E_{\calO,p,I}[N]$; in particular, as $I$ shrinks, the connected part of $E_{\calO,p,I}[N]$ continually cedes rank to the \'etale part. When $I$ is small enough, there is no more connected part, and there, any two $H$-level structures on $E_{\calO,p}$ can be distinguished from each other on $E_{\calO,p,I}$. Our algorithm below finds the largest possible $I$ such that $E_{\calO,p,I}[N]$ has no connected part.

\subsection{Computation of $v_{vert}$}
Recall that $E_{\calO, p}$ has ordinary reduction precisely when $\text{ht}(\hat{E}) = 1$, and supersingular reduction otherwise (in which case $\text{ht}(\hat{E}) = 2$). Let us first handle the ordinary case.
\begin{prop}
    If $E_{\calO, p}$ has ordinary reduction, then $v_{vert}\leq 1/(p-1)$.
\end{prop}
\begin{proof}
    We claim that the highest finite slope of $[p^m](T)$ is $1/(p-1)$. Induct on $m$. In the base case, note that the Newton polygon of $[p](T)$ has vertices $(1,1)$ and $(p,0)$; in particular, note that there are necessarily no vertices in between $(1,1)$ and $(p,0)$ because there are no proper subgroups of $\Z/p\Z$.

    For the inductive step, assume we have shown the above for $m-1$. We are thus concerning ourselves with roots of $[p^m](T)/[p^{m-1}](T)$. It is equivalent to look at the roots of $[p](T) = \alpha$, where $\alpha$ is a root of $[p^{m-1}](T)/[p^{m-2}](T)$. By the inductive hypothesis, $\nu_p(\alpha) \leq 1/(p-1)$. It follows immediately that since the Newton polygon of $[p](T) = \alpha$ has vertices $(0,\nu_p(\alpha)),(p,0)$, the slopes are bounded above by $1/p(p-1) < 1/(p-1)$.
\end{proof}
Let us next tackle the case when $E_{\FF_{\mathfrak{p}}}$ is supersingular. In this case, the Newton polygon of $[p](T)$ will necessarily have vertices at $(1,1), (p^2-p+1,r),(p^2,0)$ for some $r \in \Q \cap [0,1]$. The reason is because now there is one proper subgroup of $(\Z/p\Z)^2$, namely $\Z/p\Z$, so there is a chance that an additional vertex may appear. Nevertheless, we may apply the same induction process in the proof of the previous proposition, and we find that the highest slope comes from the edge with vertices $(1,1),(p^2-p+1,r)$. From this, we obtain the bound $v_{vert} \leq (1-r)/(p^2-p)$.

Summarizing, we have found the bounds
\begin{align*}
    v_{vert} &\leq \begin{cases}
        \frac{1}{p-1} & E_{\calO,p}~\text{has ordinary reduction} \\
        \frac{1-r}{p^2-p} & E_{\calO,p}~\text{has supersingular reduction}
    \end{cases}
\end{align*}
and where $r$ denotes the $y$-coordinate of the Newton polygon of $[p](T)$ at the $x$-coordinate $p^2-p+1$. Note that both results are independent of embeddings of $F_{R,N}$, and so $v_{vert}$ will be the same for all $\Gal(F_{R,N}/F_R)$-conjugates of $b$, i.e. $v_{d,vert} = v_{vert}$.
\subsection{Description of $e$}
Recall that $e$ denotes the ramification index of $b_p$ above the $j$-line. Letting $N =: N_0p^m$ as above, note that mod $\mathfrak{p}$ reduction of level structures induces a map $\text{pr}\colon H \hookrightarrow \GL_2(\Z/N\Z) \to G_0$, where 
\begin{align*}
    G_0 := \begin{cases}
        \GL_2(\Z/N_0\Z) \times \Hom_{\text{surj}}((\Z/p^m\Z)^2, \Z/p^m\Z) & E_{\FF_{\mathfrak{p}}}~\text{ordinary} \\
        \GL_2(\Z/N_0\Z) & E_{\FF_{\mathfrak{p}}}~\text{supersingular.}
    \end{cases}
\end{align*}
While $G_0$ may not necessarily be a group, it is still a left $H$-set and a right $A_{E_{\FF_{\mathfrak{p}}}}$-set; in particular $H$ still acts by elements of $\GL_2(\Z/N\Z)$ while $A_{E_{\FF_{\mathfrak{p}}}}$ acts by automorphisms of $E_{\FF_{\mathfrak{p}}}[N]$. Refer to e.g. $\text{pr}(H)$ as $\bar{H}$ for short, and refer similarly for elements of $\GL_2(\Z/N\Z)$.

The map $\text{pr}$ then induces a map on level structures $\text{pr}_*\colon H\backslash G/A_{E_{\calO,p}} \to \bar{H}\backslash G_0/A_{E_{\FF_{\mathfrak{p}}}}$, given by $Hg_0A_{E_{\calO,p}} \mapsto \bar{H}\bar{g_0}A_{E_{\FF_{\mathfrak{p}}}}$. Recall that our level structure for $b$ is $[\iota]_H := HgA_E$. We may thus characterize $e$ as the cardinality of the inverse image under $\text{pr}_*$ of the double coset $\bar{H}\bar{g}A_{E_{\FF_{\mathfrak{p}}}}$, that is:
\begin{prop}
    We have $e = \#\text{pr}_*^{-1}(\bar{H}\bar{g}A_{E_{\FF_{\mathfrak{p}}}})$.
\end{prop}
Note that this formula for $e$ is valid even when $\mathfrak{p}|j_E(j_E-1728)$.
\begin{exmp}
    If $\mathfrak{p}|j_E$, but $j_E \neq 0,1728$, and $p\nmid N$, then the double coset $HgA_E$ is just $Hg$ (remember that we are supposing $-I \in H$), but modulo $\mathfrak{p}$ it maps to the double coset $HgA_{E_{\FF_{\mathfrak{p}}}}$, which is three times as large as $Hg$. In particular, the cosets that reduce to $HgA_{E_{\FF_{\mathfrak{p}}}}$ are precisely $\{Hg, Hg\alpha, Hg\alpha^2\}$ for a certain order $3$ element $\alpha$.
\end{exmp}

\subsection{Computation of $e_{den}$}
We will now describe how to compute $e_{den}$. If $E_{\FF_{\mathfrak{p}}}$ is supersingular, then the computation of $e$ can be done by simply looping over all possible double cosets; since this is independent of embeddings, we will have $e = e_{den}$ in this case. 

If $E_{\FF_{\mathfrak{p}}}$ is ordinary, then $e$ is now dependent on the positioning of the mod $\mathfrak{p}$ kernel $\ker(E_{\calO,p}[N] \to E_{\FF_{\mathfrak{p}}}[N])$ relative to the basis $[\tau/N,1/N]$ of $E$. However, note that such positioning is dependent on our embeddings of $F_{R,N}$ into $\C$ and $\C_p$, and to compute $e_{den}$ we need only find the maximum across all embeddings. In particular, it is highly likely that the action of $\Gal(F_{R,N}/F_R)$ will send $\ker(E_{\calO,p}[N] \to E_{\FF_{\mathfrak{p}}}[N])$ to every possible order $p^m$ cyclic subgroup of $E[N]$. So to get an upper bound for $e_{den}$, we can compute $\#\text{pr}_*^{-1}(\bar{H}\bar{g}A_{E_{\FF_{\mathfrak{p}}}})$ for all order $p^m$ cyclic subgroups of $E[N]$, and then take the maximum.

\section{Denominators coming from horizontal ramification} \label{Primes dividing $j_E$ or $j_E - 1728$}
Suppose that $\mathfrak{p}|j_E(j_E-1728)$, and leave open the possibility that $p|N$. Let us fix $j_0$ to be either $0$ or $1728$ so that $\mathfrak{p}|j_E-j_0$. As before, let $e$, $v_{horiz}$, $v_{vert}$, $v_{den}$, $e_{den}$, $e_{d,horiz}$ and $e_{d,vert}$ be as in \cref{section:ramificationoverview}. Let $v_{a} := \nu_p(j_E-j_0)$.
\begin{rem} \label{rem:bothj0}
    The prime $\mathfrak{p}$ divides both $j_E$ and $j_E - 1728$ if and only if $p < 5$. If we encounter this scenario, then we simply perform the algorithm below for both values of $j_0$, and then take $v_{horiz}$ to be the maximum of the values found.
\end{rem}
Let $u_{j_0}$ denote a nontrivial automorphism of the elliptic curve over $\C$ with $j$-invariant $j_0$; choose the automorphism so that $u_{j_0}$ has order $n_{j_0}$ and characteristic polynomial $f_{j_0}(T)$, where
\[
    n_{j_0} := \begin{cases}
        3 & \text{if}~j_0 = 0 \\
        4 & \text{if}~j_0 = 1728,
    \end{cases} \qquad f_{j_0}(T) := \begin{cases}
        T^2+T+1 & \text{if}~j_0 = 0 \\
        T^2+1 & \text{if}~j_0 = 1728.
    \end{cases}
\]
\subsection{Description of $v_{d,horiz}$}
We have $v_{horiz} = v_{a}$ if $b_p$ is horizontally ramified , and $v_{horiz} = 0$ otherwise. Similarly, we have $v_{d,horiz} = v_a$ if at least one $\Gal(F_{R,N}/F_R)$-conjugate of $b_p$ is horizontally ramified, and $v_{d,horiz} = 0$ otherwise. So let us first compute $e_{den} \cdot v_{d,vert}$ as in \cref{Primes dividing $N$}. (Note that $v_{d,vert} = 0$ if $p \nmid N$.) If we find that $e_{den} \cdot v_{d,vert} \geq v_{a}$, then we immediately have $v_{den} = e_{den} \cdot v_{d,vert}$. So for the rest of this section, suppose $e_{den} \cdot v_{d,vert} < v_{a}$. Our goal is to determine $v_{d,horiz}$.
\subsection{Computation of $v_{d,horiz}$}
Checking that $v_{horiz} = 0$ amounts to showing $A_{E_{\FF_\mathfrak{p}}} \subseteq g^{-1}Hg$; therefore, by \cref{prop:mustcheckallgaloisconjugates}, checking that $v_{d,horiz} = 0$ amounts to showing $\sigma A_{E_{\FF_\mathfrak{p}}}\sigma^{-1} \subseteq g^{-1}Hg$ for all $\sigma \in \Gal(F_{R,N}/F_R)$. More concretely, we must check that for all $\sigma$, for each applicable $j_0$ as per \cref{rem:bothj0}, and for all order $n_{j_0}$ automorphisms $u_{j_0}$ of $E_{\FF_{\mathfrak{p}}}$, we have $$\sigma u_{j_0} \sigma^{-1} \subseteq g^{-1}Hg.$$

Let $j_0$ be as before. To work with $u_{j_0}$, we must determine its action on the $N$-torsion with respect to the basis $[\tau/N, 1/N]$ for $E[N]$. Since we defined the characteristic polynomial of $\tau$ to be $T^2+rT+s$, we end up with a ring homomorphism $\Z[\tau] \to M_2(\Z/N\Z)$ taking $\tau$ to the matrix $\left[\begin{smallmatrix}
    -r & 1 \\ -s & 0
\end{smallmatrix}\right]$. We would like to see how $\tau$ interacts with $u_{j_0}$, so we must work with $\End(E_{\FF_{\mathfrak{p}}})$.

If $\End(E_{\FF_{\mathfrak{p}}})$ has rank $2$ over $\Z$, then $\tau$ commutes with $u_{j_0}$, and thus there is essentially a unique choice for the matrix corresponding to $u_{j_0}$ (up to at most $\pm I$). Then $\Gal(F_{R,N}/F_R)$ centralizes $u_{j_0}$, and so to show that $v_{d,horiz} =0$ we need only check that $u_{j_0} \in g^{-1}Hg$.

Now assume $\End(E_{\FF_{\mathfrak{p}}})$ has rank $4$ over $\Z$. To wit, let $\Q_{p,\infty}$ be the quaternion algebra over $\Q$ ramified at $p$ and $\infty$, and inside $\Q_{p,\infty}$ fix a maximal order $\Z_{p,\infty}$ (such a choice is unique up to conjugation). By the work of \cite{supersingularEnd}, we may compute an identification $\End(E_{\FF_{\mathfrak{p}}}) \xrightarrow{\sim} \Z_{p,\infty}$. In this way we can identify $\tau$ as an element of $\Z_{p,\infty}$. 

Because $p$ might divide $N$, there is not necessarily a map $\End(E_{\FF_{\mathfrak{p}}}) \to M_2(\Z/N\Z)$. However, we have the following statement, which is a consequence of \cite[Prop. 2.7]{grosszagier}:
\begin{prop}\label{prop:deuringliftnilpotent}
    The endomorphisms $u_{j_0}$ and $\tau$ in $\End(E_{\FF_{\mathfrak{p}}})$ lift uniquely to endomorphisms in $\End(E_{\calO,p,\mathfrak{m}^{v_a}})$. In particular, the natural map $\End(E_{\calO,p,\mathfrak{m}^{v_a}}) \to \End(E_{\FF_{\mathfrak{p}}})$ is an injection.
\end{prop}
Moreover, because we supposed $e_{den} \cdot v_{d,vert} < v_{a}$, the $N$-torsion subgroup $E_{\calO,p,\mathfrak{m}^{v_a}}[N]$ has no connected part. As a result, we have:
\begin{prop}
    There is a natural map $\End(E_{\calO,p,(\pi^{v_a})}) \to M_2(\Z/N\Z)$.
\end{prop}
Here is our strategy. We can first solve the equation $f_{j_0}(T) = 0$ in $\Z_{p,\infty}$. Call the resulting solution set $\mathcal{S}_{j_0} \subseteq \Z_{p,\infty}$. Next, for each $u \in \mathcal{S}_{j_0}$, compute a relation between $\tau$ and $u$ inside of $\Z_{p,\infty}$, say
\[
    \alpha_0\tau u = \alpha_1 + \alpha_2\tau + \alpha_3u + \alpha_4u\tau
\]
for $\alpha_0,\dots, \alpha_4 \in \Z$. Finally, use the above relation, as well as $f_{j_0}(u) = 0$, to determine the possible matrices that $u$ can take on inside of $M_2(\Z/N\Z)$. In other words, we can plug $\tau = \left[\begin{smallmatrix}
    -r & 1 \\ -s & 0
\end{smallmatrix}\right]$ into the relation and then use the method of undetermined coefficients to determine all matrices that might correspond to $u$; call this set of matrices $M_u$. 

To summarize the last few paragraphs: letting $\mathscr{O}$ denote the $\Z[\tau]$-algebra $$\mathscr{O} := \dfrac{\Z[\tau]\left<u\right>}{(\alpha_1 + \alpha_2\tau + \alpha_3u + \alpha_4u\tau-\alpha_0\tau u,\, f_{j_0}(u))},$$ we have the commutative diagram
\[
    \begin{tikzcd}
        \mathscr{O} \arrow[r,hookrightarrow] \arrow[d,hookrightarrow]& \End(E_{\calO,p,(\pi^{v_a})}) \arrow[r] \arrow[d,hookrightarrow] & M_2(\Z/N\Z) \\
        \Z_{p,\infty} \arrow[r,"\sim"] & \End(E_{\FF_{\mathfrak{p}}}). & 
    \end{tikzcd}
\]
Denote $M_{j_0}$ the union of all $M_u$ as $u$ ranges across $M_{j_0}$. That is, define $$M_{j_0} := \bigcup\limits_{u \in \mathcal{S}_{j_0}} M_u \subseteq M_2(\Z/N\Z)$$ the set of all possibilities for $u_{j_0}$. We now claim that our computation of $M_{j_0}$ is enough to determine $v_{d,horiz}$:
\begin{prop}
    The set $M_{j_0}$ is invariant under conjugation by $\sigma \in \Gal(F_{R,N}/F_R)$.
\end{prop}
\begin{proof}
    We will in fact show that $M_u$ is invariant under conjugation, for each $u \in \mathcal{S}_{j_0}$. By CM theory, the image of $\Gal(F_{R,N}/F_R) \to \GL_2(\Z/N\Z)$ lies inside of the centralizer subgroup $\text{Cent}_{M_2(\Z/N\Z)}(\tau)$. So if we have $u' \in M_u$, then the Galois action takes it to some $u'' = \sigma u' \sigma^{-1}$. But by the centralizer condition, we also have $\tau = \sigma \tau\sigma^{-1}$, meaning that $u''$ and $\tau$ satisfy the exact same relations that $u'$ and $\tau$ did. Hence, $u'' \in M_u$, which completes the proof.
\end{proof}
\begin{cor}
    To show that $v_{d,horiz} = 0$, it is enough to show that the set $M_{j_0}$ is contained inside of $g^{-1}Hg$.   
\end{cor}
\subsection{Summary of the algorithm}
Based on the discussion above, here is the algorithm to compute $v_{den}$ if $\mathfrak{p}$ divides $j_E(j_E-1728)$.
\begin{itemize}
    \item If $p|N$, compute $e_{den} \cdot v_{d,vert}$ as in \cref{Primes dividing $N$}. If $e_{den} \cdot v_{d,vert} \geq v_a := \max\{\nu_p(j_E),\nu_p(j_E-1728)\}$, output $v_{den} = e_{den} \cdot v_{d,vert}$.
    \item Let $j_0 \in \{0,1728\}$ be such that $\mathfrak{p}|j_E-j_0$. If both $j_0 = 0$ and $j_0 = 1728$ satisfy this divisibility, then repeat the steps below for both these values.
    \item Using \cite{CmEnd}, compute the endomorphism of $E_0$ associated to $\tau$, and using the map $E_0 \to E_{\calO,p}$, identify $\tau$ as an endomorphism of $E_{\FF_{\mathfrak{p}}}$.
    \item Using \cite{supersingularEnd}, identify the endomorphism $\tau$ as an element of $\Z_{p,\infty}$.
    \item Compute $\mathcal{S}_{j_0} \subset \Z_{p,\infty}$, the solutions in $\Z_{p,\infty}$ to the characteristic polynomial $f_{j_0}(T) = 0$.
    \item For each $u \in \mathcal{S}_{j_0}$, find a relation $\alpha_1 + \alpha_2\tau + \alpha_3u + \alpha_4u\tau-\alpha_0\tau u = 0$.
    \item Letting $T^2+rT+s$ denote the characteristic polynomial of $\tau$, substitute $\tau =\left[\begin{smallmatrix}
    -r & 1 \\ -s & 0
    \end{smallmatrix}\right]$ into the relation, and then use the method of undetermined coefficients to solve for the possible solutions $M_u \subset M_2(\Z/N\Z)$.
    \item Compute $M_{j_0} = \bigcup\limits_{u \in \mathcal{S}_{j_0}} M_u$.
    \item If $M_{j_0} \subset g^{-1}Hg$, then output $v_{den} = 0$. Else, output $v_{den} = v_a$.
\end{itemize}
\section{The cases $j_E \in \{0, 1728\}$} \label{section:0or1728}
If we have chosen $b_\infty$ such that $j_E \in \{0,1728\}$, then as long as $b_\infty$ is not ramified over $\pi\colon Y_H \to \A^1$, we can still run the above algorithms. If $b_\infty$ is ramified, however, then we will have to pass to a larger cover. Here is how to do so. First, let $\widehat{H} \leq \GL_2(\hat{\Z})$ denote the inverse image of $H$ under the surjection $\GL_2(\hat{\Z}) \twoheadrightarrow \GL_2(\Z/N\Z)$, and let $\Gamma(3) \leq \GL_2(\hat{\Z})$ denote the full level $3$ subgroup. Then, instead of $\pi$, we may look at the map $\pi'\colon Y_{\widehat{H} \cap \Gamma(3)} \to Y_{\Gamma(3)} \cong \A^1$. Any modular curve containing a full level $3$ structure is representable, so coarse base change will hold (even though the level is now divisible by $3$). The $j$-invariant will now be replaced with $j^{1/3}$, a cube root of $j$. The analysis will proceed the exact same as detailed in this article, and in fact some things will get easier: since $\Gamma(3)$ is representable, all automorphism groups $A_E$ that appear will simply be trivial. Representability also implies that we need only worry about vertical ramification.

\appendix \section{Rigorous determination of an algebraic integer} \label{alg:algintegerrecovery}
In this section we detail an algorithm to determine an algebraic integer in a prescribed number field given numerical approximations to its Galois conjugates. We also give a speedup involving a discrete Fourier transform if we know the Galois group is abelian. Let $L/K$ be a degree $d$ Galois extension of number fields. Let $\{\sigma_0, \dots, \sigma_{d-1}\}$ denote an ordering of the elements of $\Gal(L/K)$; for convention we suppose $\sigma_0$ is the trivial element. Let $\iota\colon L \hookrightarrow L_v$ denote an embedding of $L$ into its completion at a fixed place $v \leq \infty$, and let $K_w \subseteq L_v$ denote the corresponding completion of $K$. We will make the following assumption:
\begin{quote}\label{quote:efficiency-assumption}
    It is efficient to recover an algebraic integer $\gamma$ in $\calO_K$ relative to the accuracy of an approximation $\tilde{\gamma}$ inside of $K_w \subseteq L_v$.
\end{quote}
In fact, our primary use case will be when $K$ is an imaginary quadratic field. So we will always have efficient approximations.
\subsection{The algorithm}
Suppose we have determined a vector $\mathbf{w} := [\gamma_0, \dots, \gamma_{d-1}]^T \in {\calO}^{\oplus d}_{L_v}$ to enough precision, and we know that the $\gamma_j$ are approximations to conjugates of some element $\gamma \in \calO_L$. Namely, suppose we have $\sigma_j(\gamma_1) \approx \gamma_j$ for each $0 \leq j < d$. Our goal is to find $\gamma$.

First, construct a normal basis $\{\sigma_0(\alpha), \sigma_1(\alpha), \dots, \sigma_{d-1}(\alpha)\}$ of $L/K$ such that their $\calO_K$-span contains $\calO_L$. Form the $d \times d$ matrix $\mathbf{A} := [A_{ij}]_{0\leq i,j < d}$ such that the entry $A_{ij}$ equals $\iota(\sigma_{i-j}(\alpha)) \in L_v$. Next, observe that if we had $\mathbf{w}_0 = [\gamma, \sigma_1(\gamma), \dots, \sigma_{d-1}(\gamma)]^T$ on the nose, then a solution $\mathbf{v}_0 := [v_0, \dots, v_{d-1}]^T \in \calO_K^{\oplus d}$ to $\mathbf{A}\mathbf{v}_0 = \mathbf{w}_0$ would express $\gamma$ as a $\calO_K$-linear combination $v_0\alpha + v_1\sigma_1(\alpha) + \cdots + v_{d-1}\sigma_{d-1}(\alpha)$ of our normal basis. So let us compute $\mathbf{v} = \mathbf{A}^{-1}\mathbf{w}$. We end up with $\mathbf{v}_0 := [\tilde{v}_0, \dots, \tilde{v}_{d-1}]^T \in L_v^{\oplus d}$, and now by our efficiency assumption, we can recover an algebraic integer $v_j \in \calO_K$ from $\tilde{v}_j$, for each $0 \leq j < d$. 

We output $\gamma = v_0\alpha + v_1\sigma_1(\alpha) + \cdots v_{d-1}\sigma_{d-1}(\alpha) \in \calO_L$.

\subsection{Speedup for abelian extensions}
If we know $G := \Gal(L/K)$ is abelian, then we may make the following speedup. First, write $G$ in Smith normal form i.e. $G = \Z/d_1\Z \oplus \Z/d_2\Z \oplus \cdots \oplus \Z/d_k\Z$, where we have the chain of divisor relations $d_1|d_2|\cdots|d_k$. For each $1 \leq i \leq k$, let $g_i$ be a generator for the direct summand $\Z/d_i\Z$. So we may form $k$-dimensional arrays with dimensions $d_1 \times d_2 \times \cdots \times d_k$. Concretely, let $\mathbf{a}$ and $\mathbf{w}_1$ denote arrays with those dimensions such that $\mathbf{a}[i_1, \dots, i_k]$ equals $(\iota \circ (i_1g_1 + \cdots + i_kg_k))(\alpha)$, and such that $\mathbf{w}_1[i_1, \dots, i_k]$ equals the computed ``approximation'' of $(i_1g_1 + \cdots + i_kg_k))(\gamma)$. Consequently, we may rewrite the equation $\mathbf{A}\mathbf{v}_0 = \mathbf{w}_0$ as $\mathbf{a} \ast \mathbf{v}_0 = \mathbf{w}_1$, where the $\ast$ denotes convolution:
\begin{align*}
    (\mathbf{x} \ast \mathbf{y})[m_1, \dots, m_k] := \sum_{i_1=0}^{d_1-1}\cdots \sum_{i_k=0}^{d_k-1} \mathbf{x}[i_1,\dots, i_k]\cdot \mathbf{y}[m_1-i_1,\dots, m_k-i_k].
\end{align*}

Fix a $d_k$-th root of unity $\mu_{d_k} \in \bar{L}_v$, and let $L_{v}(\zeta)$ denote the field extension obtained by adjoining $\mu_{d_k}$ to $L_v$. For $k' \leq k$, let $\mu_{d_{k'}} := \mu_{d_k}^{d_k/d_{k'}}$. Let $\mathcal{F}_G\colon L_v(\zeta)^{\oplus (d_1\times \cdots \times d_k)} \to L_v(\zeta)^{\oplus (d_1\times \cdots \times d_k)}$ denote the multidimensional \emph{discrete Fourier transform}, given by 
\begin{align*}
    \calF_G(\mathbf{x})[m_1, \dots, m_k] &:= \sum_{i_1=0}^{d_1-1}\cdots \sum_{i_k=0}^{d_k-1} \mathbf{x}[i_1, \dots, i_k]\mu_{d_1}^{-i_1m_1}\cdots \mu_{d_k}^{-i_km_k}.
\end{align*}
Of course, this map has an inverse:
\begin{align*}
    \calF_G^{-1}(\mathbf{X})[m_1, \dots, m_k] &:= \frac{1}{\# G}\sum_{i_1=0}^{d_1-1}\cdots \sum_{i_k=0}^{d_k-1} \mathbf{X}[i_1, \dots, i_k]\mu_{d_1}^{i_1m_1}\cdots \mu_{d_k}^{i_km_k}.
\end{align*}
Because $\calF_G$ turns convolution into pointwise multiplication, the equation $\mathbf{a} \ast \mathbf{v}_0 = \mathbf{w}_1$ turns into $\calF_G(\mathbf{a})\calF_G(\mathbf{v}_0) = \calF_G(\mathbf{w}_1)$. Thus, we have $\mathbf{v}_0 := \calF_G^{-1}\left(\calF_G(\mathbf{w}_1)/\calF_G(\mathbf{a})\right)$ where the division is done pointwise. We may recover an algebraic integer $v(m_1, \dots, m_k) \in \calO_K$ from each entry $\mathbf{v}_0[m_1, \dots, m_k]$.

We output $\gamma = \sum\limits_{i_1=0}^{d_1-1}\cdots \sum\limits_{i_k=0}^{d_k-1} v(i_1, \dots, i_k) \cdot (i_1g_1 + \cdots + i_kg_k)(\alpha) \in \calO_L.$
\bibliographystyle{amsalpha}
\bibliography{source.bib}

\end{document}